\documentclass[leqno,11pt]{amsart}
\usepackage{amssymb, amsmath, color}

\let\pa=\partial
\let\f=\frac
\let\wt=\widetilde
\let\wh=\widehat
\let\D=\Delta
\let\r=\rho
\let\vr=\varrho

\let\lam=\lambda
\def\na{\nabla}
\let\th=\theta
\def\s{\sigma}
\def\e{\epsilon}
\def\al{\alpha}
\def\ve{\varepsilon}
\def\fq{\mathfrak{q}}

\def\cA{{\mathcal A}}
\def\cB{{\mathcal B}}
\def\cC{{\mathcal C}}
\def\cF{{\mathcal F}}

\def\cL{{\mathcal L}}
\def\cM{{\mathcal M}}
\def\cO{{\mathcal O}}

\def\Z{\mathop{\mathbb Z\kern 0pt}\nolimits}
\def\N{\mathop{\mathbb N\kern 0pt}\nolimits}
\def\Q{\mathop{\mathbb Q\kern 0pt}\nolimits}
\def\R{\mathop{\mathbb R\kern 0pt}\nolimits}
\def\Supp{\mathop{\rm Supp}\nolimits\ }

\def\dive{{\mathop{\rm div}\nolimits}\,}
\def\diveh{{\mathop{\rm div}\nolimits}\,}

\def\nablah{\nabla}

\def\Dh{\Delta}

\def\PP{\mathbb{P}}

\def\dj{\Delta_j}
\def\dhk{\Delta_j}
\newcommand{\w}[1]{\langle {#1} \rangle}
\def\alp{\alpha_{\sigma,p}}
\def\atp{\alpha_{2,p}}

\def\dB{\dot{B}}

\newcommand{\andf}{\quad\hbox{and}\quad}
\newcommand{\with}{\quad\hbox{with}\quad}
\def\eqdef{\buildrel\hbox{\footnotesize def}\over =}

 \numberwithin{equation}{section}

\newtheorem{defi}{Definition}[section]
\newtheorem{thm}{Theorem}[section]
\newtheorem{lem}{Lemma}[section]
\newtheorem{rmk}{Remark}[section]

\begin{document}

\title[The optimal time-decay estimates for 2-D inhomogeneous N-S]
{The optimal time-decay estimates for 2-D inhomogeneous Navier-Stokes equations}

\author[Y. Liu]{Yanlin Liu}
\address[Y. Liu]{School of Mathematical Sciences,
Laboratory of Mathematics and Complex Systems,
MOE, Beijing Normal University, 100875 Beijing, China.}
\email{liuyanlin@bnu.edu.cn}

\date{\today}

\begin{abstract}
In this paper, we derive the optimal time-decay estimates for 2-D inhomogeneous Navier-Stokes equations. In particular, we prove that
$\|u(t)\|_{\dB^{\th}_{p,1}(\R^2)}=\cO
(t^{\f1p-\f32-\f{\th}2})$ as $t\rightarrow\infty$ for any $p\in[2,\infty[,~\th\in [0,2]$ if initially $\r_0u_0\in \dB^{-2}_{2,\infty}(\R^2)$.
This is optimal even for the classical homogeneous Navier-Stokes equations. Different with Schonbek and Wiegner's Fourier splitting device, our method here seems more direct, and can adapt to many other equations as well. Moreover, our method allows us to work in the $L^p$-based spaces.
\end{abstract}

\maketitle



\setcounter{equation}{0}

\section{Introduction}\label{sec1}
In this paper, we investigate the time-decay estimates for the solutions to the
2-D inhomogeneous incompressible Navier-Stokes equations (IN-S for short):
\begin{equation}\label{1.1}
\quad \left\{\begin{array}{l}
\displaystyle \pa_t \rho+u\cdot\na \rho=0,
\qquad\qquad (t,x)\in\R^+\times\R^2,\\
\displaystyle \rho(\pa_t u +u\cdot\nabla u)-\Dh u+\nablah P=0,\\
\displaystyle \diveh u = 0, \\
\displaystyle \rho|_{t=0}=\rho_0,\quad u|_{t=0}=u_0,
\end{array}\right.
\end{equation}
where $\rho$ and $u=(u^1,u^2)$ stand for the density and velocity of the fluid respectively, and $P$ for the scalar pressure function. Such a system can be used to describe a fluid which is incompressible but whose density can be non-constant.
\smallskip

For the special case when $\rho$ is a constant,
without loss of generality, we may assume $\rho=1$, then \eqref{1.1} reduces to be the classical Navier-Stokes equations (N-S for short):
\begin{equation}\label{1.2}
\left\{\begin{array}{l}
\displaystyle \pa_t u+u\cdot\nabla u-\Delta u=-\nabla P,\\
\displaystyle \dive u = 0, \\
\displaystyle u|_{t=0}=u_0.
\end{array}\right.
\end{equation}

Leray proved in the seminal paper \cite{lerayns} that given any finite energy initial data, N-S admits global-in-time weak solutions which verify the energy inequality:
$$\|u\|_{L^\infty_t(L^2(\R^d))}^2+2\|u\|_{L^2_t(L^2(\R^d))}^2
\leq\|u_0\|_{L^2(\R^d)}^2,\quad\forall\ t>0,~d=2\text{ or }3.$$
We mention that for 3-D N-S, the regularity and uniqueness of such weak solutions is still a fundamental open problem in mathematical fluid mechanics. While for the 2-D case, we do have a positive answer to this problem long ago, see for instance \cite{Lady59,LP59,Serrin62}.

Leray \cite{lerayns} also posed the problem of determining whether or not the weak solutions decay to zero in $L^2(\R^d)$ as time tends to infinity. There are numerous works concerning this, here we only list some.
By using Fourier transform and the method of splitting the integration in phase-space into two time-dependent domains,
Schonbek \cite{Sch85} proved that for $u_0\in L^1(\R^3)\cap L^2(\R^3)$, there exists one Leray weak solution\footnote{Actually the solutions she considered are the so-called suitable weak solutions, see \cite{CKN}.} to 3-D N-S such that
$$\|u(t)\|_{L^2(\R^3)}\leq C_{\rm in}(1+t)^{-\f14},$$
where $C_{\rm in}$ denotes some constant depending only on the norms of the initial data.

Motivated by Schonbek's Fourier splitting device,
Wiegner \cite{Wiegner} found that there is a strong connection
between the decay properties of the solutions to N-S and to
the heat equation with the same initial data.
Precisely, he proved that for $u_0\in L^2(\R^d)$
with $\|e^{t\D}u_0\|_{L^2(\R^d)}\leq C(1+t)^{-\f\al2}$
with $0<\al\leq\f{d+2}2$, which is equivalent to say that $u_0\in L^2(\R^d)\cap \dB^{-\al}_{2,\infty}(\R^d)$,
there exists one Leray weak solution such that
$$\|u(t)\|_{L^2(\R^d)}\leq C_{\rm in}(1+t)^{-\f\al2}.$$
In particular, when $\al=\f{d+2}2$ with $d=2$ or $3$, we have
\begin{equation}\label{optimaldecay}
\|u(t)\|_{L^2(\R^2)}\leq C_{\rm in}(1+t)^{-1},\andf
\|u(t)\|_{L^2(\R^3)}\leq C_{\rm in}(1+t)^{-\f54}.
\end{equation}

Later, Miyakawa and Schonbek \cite{MS01, Sch91} also studied the lower bounds of the decay rates, and proved that the decay rates in \eqref{optimaldecay} are optimal.

On the other hand, since the 1980s', a number of works have been dedicated to the study of optimal
time-decay estimates in fluid mechanics, and many of them are not related to
Shonbeck and Wiegner’s method. Here we only list \cite{GW12, MN80, XK15, XK17} concerning the compressible Navier-Stokes equations, or the partially dissipative hyperbolic systems for instance.

\smallskip

The aim of this paper is to find a new approach to
establish these optimal decay estimates,
not only for the classical N-S,
but also for the inhomogeneous one.

Our main result states as follows:

\begin{thm}\label{thm1}
{\sl Let $p\in[2,\infty[,~\sigma\in]0,2]$,
$\rho_0\in\dB^1_{2,1}\cap \dB^{2-\f2p}_{2,1}(\R^2)$
and $u_0\in \dB^0_{2,1}\cap \dB^0_{p,1}(\R^2)$
with $\r_0u_0\in \dB^{-\sigma}_{2,\infty}(\R^2)$
and $\dive u_0=0~$\footnote{See Section \ref{sec2} for the definition of the functional spaces.}.
There exists a small positive constant $\e$ so that if
\begin{equation}\label{small}
\|\rho_0-1\|_{\dB^1_{2,1}(\R^2)}
\exp\bigl(C\|u_0\|_{\dB^0_{2,1}(\R^2)}
\exp(C\|u_0\|_{L^2(\R^2)}^4)\bigr)
\leq \e,
\end{equation}
then the system \eqref{1.1} has a unique global solution
$u\in C([0,\infty[;\dB^0_{2,1}(\R^2))\cap L^1(\R^+; \dB^2_{2,1}(\R^2))$,
which satisfies for any $t>0$ and any $\s\in]0,2[$ that
\begin{equation}\label{SAeq2}
\|u(t)\|_{\dB^0_{2,1}(\R^2)}
\lesssim\alp \langle t\rangle^{-\f \sigma2},
\quad \|u(t)\|_{\dB^{\th}_{p,1}(\R^2)}
\lesssim\alp^6
\langle t\rangle^{\max\{\f1p-\f12-\f\sigma2,-\s\}}t^{-\f{\th}2},
\quad\forall\ \th\in [0,2],
\end{equation}
where $\langle t\rangle$ denotes $1+t$, and for $\s=2$ that
\begin{equation}\label{1.7}
\|u(t)\|_{\dB^0_{2,1}(\R^2)}\lesssim \atp^{12}\w{t}^{-1},
\quad
\|u(t)\|_{\dB^{\th}_{p,1}(\R^2)}\lesssim\atp^{72}
\langle t\rangle^{\f1p-\f32}
t^{-\f{\th}2},\quad\forall\ \th\in [0,2],
\end{equation}
where $\alp\eqdef\|\rho_0u_0\|_{\dB^0_{2,1}\cap \dB^0_{p,1}\cap \dB^{-\sigma}_{2,\infty}(\R^2)}
\exp\bigl(\exp(C\|u_0\|_{\dB^0_{2,1}(\R^2)}^4)\bigr)$
and $\alp^s\eqdef\alp(1+\alp)^{s-1}$.
}\end{thm}

\begin{rmk}\label{rmk1.1}
By using Schonbek and Wiegner's Fourier splitting device,
Chemin and Zhang \cite{CZ6} also derived the optimal
$L^2(\R^2)$ decay estimates for IN-S \eqref{1.1}.

However, our method here is different.
We use Littlewood-Paley theory and Bony's decomposition
instead of the Plancherel's identity and Fourier splitting device used in \cite{CZ6,Sch85,Wiegner} respectively.
In particular, these release us from working in the $L^2$-based spaces.
Actually our method works for the $L^p$-based spaces as well.
Moreover, the method here seems more direct,
and can be used in the estimates of many other equations as well.
\end{rmk}

\begin{rmk}
It is a long-standing open problem whether 3-D N-S is
globally well-posed or not for large initial data.
Chemin and Gallagher \cite{CG10} considered one enlightening case:
\begin{equation}\label{CGslow}
u_0(x)=(U^1_0 +\ve V^1_0,U^2_0 +\ve V^2_0,V^3_0)(x_1,x_2,\ve x_3),
\with\dive U_0=\dive V_0=0,
\end{equation}
where $\ve$ is small positive constant
depending on some norms of $U_0$ and $V_0$.
The initial data in \eqref{CGslow} is large in
$\dB^{-1}_{\infty,\infty}(\R^3)$,
but still can generate global strong solution to 3-D N-S.
Later, this kind of initial data was studied in a series of works,
see for instance \cite{CGZ, CZ15, LZ4,LPZ}. We mention that,
the main idea in these works is to construct an approximating solution,
which formally satisfies 2-D N-S but with $x_3$ as a parameter.

Motivated by \cite{CG10}, it is natural to consider 3-D IN-S,
with initial velocity in the same form as \eqref{CGslow},
and initial density close to some constant.
In this case, it is reasonable to expect that we can use 2-D IN-S with a parameter $x_3$ as the approximating system.

On the other hand, as we know,
in order to preserve the regularity of the density,
we need $L^1_t({\rm Lip})$ estimate of the velocity.
Hence a comprehensive study on the time-decay estimates for
2-D IN-S could be an important ingredient in the study of
this approximating system, as well as the 3-D IN-S
with initial velocity in the form \eqref{CGslow}.
Moreover, as we have mentioned in Remark \ref{rmk1.1},
here we can also obtain time-decay estimates in the $L^p$-based Besov spaces,
which can be crucial when considering high-oscillating initial data.
\end{rmk}

Let us end this section with some notations that we shall use throughout this paper.

\noindent{\bf Notations:}
We always use $C$ to denote an absolute constant
which may vary from line to line, and $a\lesssim b$ means that $a\leq Cb$.
For a Banach space $B$, we shall use the shorthand $L^p_T(B)$ and $L^p_{T,f}(B)$ for $\bigl\|\|\cdot\|_B\bigr\|_{L^p(0,T;dt)}$
and $\bigl\|\|\cdot\|_B\bigr\|_{L^p(0,T;f(t)dt)}$ respectively.

\setcounter{equation}{0}
\section{Functional spaces and some technical lemmas}\label{sec2}

In this section, we shall introduce the functional spaces used in this paper,
and collect some basic facts on Littlewood-Paley theory.
We would like to emphasize that in the rest of this paper,
all the functional spaces are all defined on $\R^2$ for space variables.

Let us first recall the dyadic operators from \cite{BCD}:
$$\Delta_ja\eqdef \cF^{-1}(\varphi(2^{-j}|\xi|)\widehat{a}),\andf S_ja\eqdef\cF^{-1}(\chi(2^{-j}|\xi|)\widehat{a})$$
where $\chi(\tau)$ and $\varphi(\tau)$ are smooth functions such that
\begin{align*}
&\Supp \varphi \subset \bigl\{\tau \in \R\,: \, 3/4 \leq
|\tau| \leq 8/3 \bigr\}\quad\mbox{and}\quad \forall
 \tau>0\,,\ \sum_{j\in\Z}\varphi(2^{-j}\tau)=1,\\
& \Supp \chi \subset \bigl\{\tau \in \R\,: \, |\tau| \leq
4/3 \bigr\}\quad\mbox{and}\quad \forall
 \tau\in\R\,,\ \chi(\tau)+ \sum_{j\geq 0}\varphi(2^{-j}\tau)=1.
\end{align*}

\begin{defi}\label{defbesov}
{\sl Let $p,r\in[1,\infty]$ and $s\in\R$. The homogeneous Besov space $\dB^s_{p,r}$ consists of those $a\in{\mathcal S}'$ with $\lim_{j\to-\infty}\|S_ja\|_{L^\infty}=0$
such that
$$\|a\|_{\dB^s_{p,r}}\eqdef\big\|\big(2^{js}
\|\Delta_j a\|_{L^p}\big)_{j\in\Z}\bigr\|
_{\ell ^{r}(\Z)}<\infty.$$
And we shall use $\cB^s_p$ to denote $\dB^{s}_{p,1}$ for notation simplification.
}\end{defi}

\begin{defi}
{\sl The Chemin-Lerner type norm (see \cite{CL}) is given by
$$\|a\|_{\wt{L}^q_T(\dB^s_{p,r})}
\eqdef \big\|\big(2^{js}\|\Delta_j a\|
_{L^q_T(L^p)}\big)_{j\in\Z}\bigr\|
_{\ell ^{r}(\Z)}.$$
And for any $f\in L^1_{\rm{loc}}$ with $f\geq 0$,
we recall the following time-weighted norm from \cite{PZ1}:
$$
\|a\|_{\widetilde L^q_{T,f}(\dB^s_{p,r})}
\eqdef\Bigl\|\Bigl( 2^{js}
\bigl(\int_0^T
\|\dj a(t)\|_{L^p}^qf(t)\,dt\bigr)^{\frac1q}\Bigr)
_{j\in\Z}\Bigr\|_{\ell ^{r}(\Z)}.$$
}\end{defi}

The following two Bernstein-type inequalities will play an essential role in this paper.

\begin{lem}[\cite{BCD}]\label{lemBern}
{\sl Let $\cB$ be a ball of $\R^2$, and $\cC$ a ring of $\R^2$. Then for any $1\leq p_2\leq p_1\leq\infty$ and any $m\in\N$, there holds:
\begin{align*}
&\Supp \wh a\subset 2^j\cB
\Rightarrow\|\nabla^m a\|_{L^{p_1}}
\lesssim 2^{j\left(m+2/{p_2}-2/{p_1}\right)}
\|a\|_{L^{p_2}};\\
&\qquad\Supp\wh a\subset 2^j\cC
\Rightarrow\|a\|_{L^{p_1}} \lesssim
2^{-jm}\|\nabla^m a\|_{L^{p_1}}.
\end{align*}
}
\end{lem}

\begin{lem}[\cite{Dan01}]\label{lemdan}
{\sl Let $1<p<\infty$ and $\Supp\wh u\in\{\xi\in\R^2:\ R_1<|\xi|<R_2\}$.
 Then there exists a constant $c$ depending only on $R_2/R_1$ such that
$$c\f{R_1^2}{p^2}\int_{\R^2}|u|^p\,dx
\leq\int_{\R^2}|\nabla u|^2|u|^{p-2}\,dx
=-\f{1}{p-1}\int_{\R^2}\D u|u|^{p-2}u\,dx.$$
}\end{lem}

The following result concerns the action of smooth functions on Besov spaces.
\begin{lem}[\cite{Meyer}]\label{comBesov}
{\sl Let $s>0$ and $p,~r\in[1,\infty]$. Let $F\in W^{[s]+2,\infty}_{\rm{loc}}$
such that $F(0)=0$. Then there exists some constant $c$ depending only on $s,~F$
and $\|u\|_{L^\infty}$ such that
$$\|F\circ u\|_{\dB^s_{p,r}}
\leq c\|u\|_{\dB^s_{p,r}}.$$
}\end{lem}

And the following product law will be frequently used in this paper.

\begin{lem}[\cite{BCD}]\label{lemproductlaw}
{\sl For any $p\in[2,\infty[$, any $s_1,s_2\in]-2/p,2/p[$ with $s_1+s_2>0$, and any $r,r_1,r_2\in[1,\infty]$
with $1/r=1/r_1+1/r_2$, there holds
$$\|ab\|_{\dB^{s_1+s_2-\f2p}_{p,r}}
\lesssim\|a\|_{\dB^{s_1}_{p,r_1}}
\|b\|_{\dB^{s_2}_{p,r_2}},$$
and for the borderline case that
$$\|ab\|_{\dB^{s_1}_{p,r}}
\lesssim\|a\|_{\dB^{s_1}_{p,r}}
\|b\|_{\cB^{\f2p}_{p}},\andf
\|ab\|_{\dB^{-\f2p}_{p,\infty}}
\lesssim\|a\|_{\dB^{-\f2p}_{p,\infty}}
\|b\|_{\cB^{\f2p}_{p}}.$$
}\end{lem}

The following lemma concerns a general system,
whose proof will be put in Appendix \ref{appA}.

\begin{lem}\label{lem4.1}
{\sl Let  $b_1,b_2,\eta:[0,T]\times\R^2\rightarrow\R^2$ be smooth vector fields
with $\diveh b_1=\diveh b_2=0.$ Then under the assumption that
$\|a\|_{{L}^\infty_t(\cB^1_{2})}$ is sufficiently small,
for any $p\in[2,\infty[$ and any $s\in ]-2/p, 2/p],$ the following system
\begin{equation}\label{eqt2dw}
\left\{\begin{array}{l}
\displaystyle \pa_t w +b_1\cdot\nablah w
+w\cdot\nablah b_2-(1+a)(\Dh w-\nablah\Pi)=\eta,\\
\displaystyle \diveh w = 0, \\
\displaystyle w|_{t=0}=w_0,
\end{array}\right.
\end{equation}
has a unique solution so that for any $t\leq T$, there holds
\begin{equation}\begin{split}\label{lem4.1a}
\|w\|_{\wt{L}^\infty_t(\cB^s_{p})}+\|\Dh w\|_{L^1_t(\cB^s_{p})}
\leq&C\bigl(\|w_0\|_{\cB^s_{p}}
+\|\eta\|_{L^1_t(\cB^s_{p})}\bigr)\\
&\times\exp\bigl(C\|b_1\|_{L^2_t(\cB^{\f2p}_{p})}^2
+C\|\nablah b_2\|_{L^1_t(\cB^{\f2p}_{p})}\bigr).
\end{split}\end{equation}
Furthermore, if $b_2=\eta=0$ and $s\in ]-2/p,1/p]$, there holds
\begin{equation}\begin{split}\label{lem4.1b}
\|w\|_{\wt{L}^\infty_t(\cB^s_{p})}+\|\Dh w\|_{L^1_t(\cB^s_{p})}
\leq C\|w_0\|_{\cB^s_{p}}
\exp\bigl(C\|b_1\|_{L^\infty_t(L^2)}^{\f2{p-1}}
\|\nablah b_1\|_{L^2_t(L^2)}^2\bigr).
\end{split}\end{equation}
}\end{lem}

At the end of this section, let us present
two useful lemmas in deriving time-decay estimates.

\begin{lem}\label{Salem3}
{\sl Let $q\in[1,\infty[$ and $f(t)\in C[0,\infty[$ satisfy
$0\leq f(t)\leq C_0$ and $\|f\|_{L^{q}(0,t)}\leq C_1$ for any $t\in\R^+,$
and $f(t)\leq C_2f(\tau)+C_1C_2 \tau^{-\f1q}$
for any $t\geq \tau\geq 0.$ Then one has
\begin{equation}\label{decayft}
f(t)\leq \max\left\{2C_0,8C_1C_2\right\}\langle t\rangle^{-\f1q}.
\end{equation}
}\end{lem}
\begin{proof} For any fixed $t>0,$ due to $\|f\|_{L^q(t/2,t)}\leq C_1,$
there exists $t_1\in ]t/2,t[$ so that
$f(t_1)\leq 2^{\f1q}C_1t^{-{\f1q}}.$ As a result, it comes out
\begin{align*}
f(t)\leq \min\bigl\{ C_0,C_2f(t_1)
+C_1C_2 |t_1|^{-\f1q} \bigr\}
\leq \min\bigl\{ C_0,2^{1+\f1q}C_1C_2 t^{-\f1q} \bigr\}.
\end{align*}
In view of this, for $t\leq1$, there holds $\langle t\rangle\leq2$, and thus
$$f(t)\leq C_0\leq 2^{\f1q}C_0\langle t\rangle^{-\f1q}\leq2C_0\langle t\rangle^{-\f1q}.$$
While for $t>1$, there holds $\langle t\rangle\leq2t$, and thus
$$f(t)\leq 2^{1+\f1q}C_1C_2 t^{-\f1q}
\leq 2^{1+\f2q}C_1C_2 \langle t\rangle^{-\f1q}
\leq8C_1C_2 \langle t\rangle^{-\f1q}.$$
Combining the above two situations leads to the desired estimate \eqref{decayft}.
\end{proof}

\begin{lem}\label{lemit}
{\sl Let $f:\R^+\rightarrow\R^+$ satisfy $f(t)\leq C_3$ whenever $t\in[1/2,1[$, and
\begin{equation}\label{it1}
f(t)\leq\ve f(t/2)+C_4 t^{-K},\quad\forall\ t\geq1,
\end{equation}
for some positive constants $\ve$ and $K$ satisfying $2^{K}\ve\leq1/2$. Then there holds
\begin{equation}\label{it2}
f(t)\leq(C_3+2C_4)t^{-K},\quad\forall\ t\geq1.
\end{equation}
}\end{lem}
\begin{proof}
For any $t\geq1$, there exists a unique integer $n$ such that
$1/2\leq2^{-(n+1)}t<1$. For each $i=0,~1,~\cdots~,n$, in view of \eqref{it1}, we deduce that
$$\ve^i f(2^{-i}t)\leq\ve^{i+1} f(2^{-(i+1)}t)+C_4\ve^i 2^{iK} t^{-K}.$$
By summing up the above inequalities for $i$ varying from $0$ to $n,$  we infer
\begin{equation}\begin{split}\label{it3}
f(t)&\leq\ve^{n+1} f(2^{-(n+1)}t)+C_4 t^{-K}\sum_{i=0}^n(2^K\ve)^i\\
&\leq\ve^{n+1} C_3+2C_4t^{-K},
\end{split}\end{equation}
where we used $2^K\ve\leq1/2$ in the last step.
Then \eqref{it2} follows from \eqref{it3} and the fact that
$$\ve^{n+1}<2^{-(n+1)K}<t^{-K}.$$
This completes the proof of this lemma.
\end{proof}

\setcounter{equation}{0}
\section{Decay estimates of $\|u(t)\|_{\cB^0_2}$ when $\s\in]0,2[$}

For completion, let us first give a brief proof that the system \eqref{1.1}
is globally well-posed under the assumptions of Theorem \ref{thm1}. As we are considering the case that $\rho_0$ is close to $1$, it is convenient to introduce $a=\rho^{-1}-1$, and rewrite \eqref{1.1} in the following form
\begin{equation}\label{2dua}
\left\{\begin{array}{l}
\displaystyle \pa_t a+u\cdot\nabla a=0,
\qquad\qquad\qquad\qquad (t,x)\in\R^+\times\R^2,\\
\displaystyle \pa_tu+u\cdot\nabla u-\Delta u=-(1+a)\nabla P+a\D u,\\
\displaystyle \dive u = 0, \\
\displaystyle a|_{t=0}=a_0=\rho_0^{-1}-1,\quad u|_{t=0}=u_0,
\end{array}\right.
\end{equation}
And in view of Lemma \ref{comBesov}, the assumptions $\rho_0-1\in\cB^1_{2}\cap \cB^2_{2}$ and \eqref{small} imply that
\begin{equation}\label{smallc0}
a_0\in\cB^1_{2}\cap \cB^2_{2},\andf
\|a_0\|_{\cB^1_{2}\cap\cB^2_2}
\exp\bigl(C\|u_0\|_{\cB^0_{2}}
\exp(C\|u_0\|_{L^2}^4)\bigr)
\leq \e_0,
\end{equation}
where $\e_0$ is some small constant depending on the $\e$ in \eqref{small}.

To begin with, let us introduce $T^\ast$ to be
\begin{equation}\begin{split}\label{aB12}
T^\ast\eqdef\sup\bigl\{T>0:
\|a\|_{\wt{L}^\infty_T(\cB^1_{2})}
\leq2\|a_0\|_{\cB^1_{2}}
\exp\bigl(C\|u_0\|_{\cB^0_{2}}\exp(C\|u_0\|_{L^2}^4)\bigr)\bigr\}.
\end{split}\end{equation}
Noticing that the smallness condition \eqref{smallc0}
guarantees that for any $t<T^\ast$,
$\|a\|_{\wt{L}^\infty_t(\cB^1_{2})}$
is still sufficiently small,
thus we can deduce from \eqref{lem4.1b} in Lemma \ref{lem4.1} that
\begin{equation}\begin{split}\label{uB02}
\|u\|_{\wt{L}^\infty_t(\cB^0_{2})}+\|\Dh u\|_{L^1_t(\cB^0_{2})}
&\leq2\|u_0\|_{\cB^0_{2}}
\exp\bigl(C\|u\|_{L^\infty_T(L^2)}^2
\|\nablah u\|_{L^2_T(L^2)}^2\bigr)\\
&\leq2\|u_0\|_{\cB^0_{2}}
\exp\bigl(C\|u_0\|_{L^2}^4\bigr),
\end{split}\end{equation}
where in the last step, we have used the energy inequality:
\begin{equation}\label{Saeq2a}
\f12\int_{\R^2}\r|u|^2\,dx
+\int_0^t\|\na u(t')\|_{L^2}^2\,dt'
\leq\f12\int_{\R^2}\r_0|u_0|^2\,dx,
\end{equation}
and the fact that $\|\rho-1\|_{\cB^1_{2}}
\lesssim\|a\|_{\cB^1_{2}}$ is small and $\cB^1_{2}\hookrightarrow L^\infty$, so that $\|\sqrt\rho-1\|_{L^\infty}
=\|\f{\rho-1}{\sqrt\rho+1}\|_{L^\infty}
\leq\|\rho-1\|_{L^\infty}$ is also small,
and thus
$$\|u\|_{L^2}
\leq\|\sqrt\rho u\|_{L^2}
+\|\sqrt\rho-1\|_{L^\infty}\|u\|_{L^2}
\Rightarrow C^{-1}\|u\|_{L^2}
\leq\|\sqrt\rho u\|_{L^2}\leq C\|u\|_{L^2}.$$

Then by using Theorem $3.14$ of \cite{BCD} and the estimate \eqref{uB02}, we infer
\begin{align*}
\|a\|_{\wt{L}^\infty_t(\cB^1_{2})}
&\leq\|a_0\|_{\cB^1_{2}}
\exp\bigl(C\|\nablah u\|_{L^1_t(\cB^1_{2})}\bigr)\\
&\leq\|a_0\|_{\cB^1_{2}}
\exp\bigl(C\|u_0\|_{\cB^0_{2}}
\exp(C\|u_0\|_{L^2}^4)\bigr),
\quad\forall\ t<T^\ast,
\end{align*}
which contradicts to \eqref{aB12}, unless $T^\ast=\infty$.
Thus there must be $T^\ast=\infty$, and
\begin{equation}\begin{split}\label{a2}
&\|u\|_{\wt{L}^\infty(\R^+;\cB^0_{2})}+\|\Dh u\|_{L^1(\R^+;\cB^0_{2})}
\leq2\|u_0\|_{\cB^0_{2}}\exp\bigl(C\|u_0\|_{L^2}^4\bigr),\\
\|&a\|_{\wt{L}^\infty(\R^+;\cB^1_{2})}
\leq2\|a_0\|_{\cB^1_{2}}
\exp\bigl(C\|u_0\|_{\cB^0_{2}}\exp(C\|u_0\|_{L^2}^4)\bigr)
\leq2\e_0.
\end{split}\end{equation}

Now let us turn back to the decay estimates for $\|u(t)\|_{\cB^0_2}$.
One important step is to rewrite the momentum equation of \eqref{1.1}.
Precisely, we first write
\begin{equation}\label{alem1}
\pa_tu-\Dh u+\nablah P=-\pa_t(\vr u)-\diveh(\r u\otimes u)\quad\mbox{with}\quad
\varrho=\rho-1.
\end{equation}
Then for any $t>0$, we use heat semigroup to write \eqref{alem1} as
\begin{align*}
u(t)=e^{t\D}u_0
-\int_0^te^{(t-t')\D}\PP\pa_t
(\vr u)(t')\,dt'
-\int_0^te^{(t-t')\D}\PP
\bigl(\diveh(\r u\otimes u)\bigr)(t')\,dt',
\end{align*}
where $\PP\eqdef {\rm Id}-\nabla \D^{-1}\dive$ is the Leray projector into divergence-free vector fields.
By using integration by parts, we find
\begin{align*}
\int_0^te^{(t-t')\D}\PP\pa_t(\vr u)(t')\,dt'=\PP(\vr u)(t)-e^{t\D}\PP(\vr_0 u_0)
+\int_0^te^{(t-t')\D}\D\PP(\vr u)(t')\,dt'.
\end{align*}
As a result, it comes out
\begin{equation}\begin{split}\label{alem2}
u(t)=e^{t\Dh}(\rho_0u_0)-\PP(\vr u)(t)&-\int_0^te^{(t-t')\Dh}\Dh \PP(\vr u)(t')\,dt'\\
&-\int_0^te^{(t-t')\Dh}\PP\bigl(\diveh(\r u\otimes u)\bigr)(t')\,dt'.
\end{split}\end{equation}

\begin{lem}\label{lemapp}
{\sl Let $p\in[2,\infty[$ and $(\rho,u)$ be a regular enough solution to \eqref{alem2} with $\|\varrho\|_{\wt{L}^\infty_t(\cB^{2/p}_{p})}$ sufficiently small.
Let $h$ and $g$ be some regular enough functions, and $v$ be determined by~\footnote{For the use of this paper, one can simply take $h=g=0$. The reason for putting an additional $h$ and $g$ in \eqref{Saq1} is for future use.
Actually, these terms will appear if we study the derivatives of $u$ in \eqref{alem2}.}
\begin{equation}\begin{split}\label{Saq1}
v(t)=e^{t\Dh}(\rho_0v_0)-\PP(\vr v+h)(t)&-\int_0^te^{(t-t')\Dh}\Dh \PP(\vr v+h)(t')\,dt'\\
&-\int_0^te^{(t-t')\Dh}\PP\bigl(\dive(\r u\otimes v+\r v\otimes u+g)\bigr)(t')\,dt'.
\end{split}\end{equation}
Then for any $s\in]1-2/p,2[,~r\in[1,\infty]$, $\fq \in]1,2[$ and $q\in[1,\infty]$ with $q,\,\fq>(\f1p+\f{s}2)^{-1}$ and
$\f1q+\f1p>\f{s}2$ (or $\f1q+\f1p=\f{s}2$ if $r=\infty$), there holds
\begin{equation}\label{Saq2}
\|v\|_{\wt L^q_t(\dB^{\f2q-s}_{p,r})}
\leq C\Bigl(\|\rho_0v_0\|_{\dB^{-s}_{p,r}}
+\|h\|_{\wt L^q_t(\dB^{\f2q-s}_{p,r})\bigcap
\wt{L}^{\fq}_t(\dB^{\f2{\fq}-s}_{p,r})}
+\|g\|_{\wt L^1_t(\dB^{1-s}_{p,r})}\Bigr)
\exp\bigl(C\|u_0\|_{L^2}^{\f{\fq}{\fq-1}}\bigr).
\end{equation}
While for the borderline case that $s=2$, we have for any $q\in [1,\infty]$ with $\f1q+\f1p\geq1$ that
\begin{equation}\label{endpoint}
\|v\|_{\wt L^q_t(\dB^{\f2q-2}_{p,\infty})}
\leq C\Bigl(\|\rho_0v_0\|_{\dB^{-2}_{p,\infty}}
+\|h\|_{\wt L^q_t(\dB^{\f2q-2}_{p,\infty})}
+\|g\|_{\wt L^1_t(\dB^{-1}_{p,\infty})}
+\|u\|_{L^2_t(L^2)}\|v\|_{L^2_t(L^p)}\Bigr).
\end{equation}
}\end{lem}


Let us postpone the proof of Lemma \ref{lemapp} till the end of this section, and continue the decay estimates of $\|u(t)\|_{\cB^0_2}$. To verify the conditions of Lemma \ref{lemapp}, we can use Lemma \ref{comBesov} and \eqref{a2} to deduce that $\vr=\rho-1=-\f{a}{1+a}$ is sufficiently small, namely
$$\|\vr\|_{\wt{L}^\infty(\R^+;\cB^1_{2})}
\leq C\|a\|_{\wt{L}^\infty(\R^+;\cB^1_{2})}
\leq C\e_0.$$
Then we can use Lemma \ref{lemapp} with $h=g=0$ and $p=2$ to derive for any $\s\in]0,2[$ that
\begin{equation}\label{SAeq1}
\|u\|_{\wt L^q_t(\dB^{\f2q-\s}_{2,\infty})}
\leq C\|\rho_0 u_0\|_{\dB^{-\s}_{2,\infty}}
\exp\bigl(C\|u_0\|_{L^2}^4\bigr),\quad\forall\ q\in \Omega_\s,
\end{equation}
where $\Omega_\s=]\f2{1+\s},\infty]$ when $\s\in]0,1]$, $\Omega_\s=[1,\f2{\s-1}]$ when $\s\in]1,2[$,
and for $\s=2$ that
\begin{equation}\label{3.11}
\|u\|_{\wt L^q_t(\dB^{\f2q-2}_{2,\infty})}
\leq C\bigl(\|\rho_0u_0\|_{\dB^{-2}_{2,\infty}}
+\|u\|_{L^2_t(L^2)}^2\bigr),
\quad\forall\ q\in [1,2].
\end{equation}

Let us first focus on the case when $\s\in]0,2[$.
Noticing that for any $0<\varsigma<\min\{\s,2-\s\}$, we have
$\f{2}{\sigma\pm\varsigma}\in\Omega_\s$.
Then we can use \eqref{SAeq1} with $q=\f{2}{\sigma\pm\varsigma}$ to get
\begin{align*}
\|u\|_{L_t^{\f{2}{\sigma}}(\cB^0_{2})}
\leq\|u\|_{\wt L_t^{\f{2}{\sigma}}(\cB^0_{2})}
\leq&C\|u\|_{\wt L_t^{\f{2}{\sigma-\varsigma}}
(\dB^{-\varsigma}_{2,\infty})}^{\f12}
\|u\|_{\wt L_t^{\f{2}{\sigma+\varsigma}}
(\dB^{\varsigma}_{2,\infty})}^{\f12}\\
\leq& C\|\r_0u_0\|_{\dB^{-\s}_{2,\infty}}
\exp\bigl(C\|u_0\|_{L^2}^4\bigr)\leq C\alp,
\end{align*}
which together with \eqref{a2} and
Lemma \ref{Salem3} ensures that
\begin{equation}\begin{split}\label{3.12}
\|u(t)\|_{\cB^0_{2}}\lesssim \alp \langle t\rangle^{-\f{\sigma}2},
\quad\forall\ \s\in]0,2[.
\end{split}\end{equation}

While the case when $\s=2$ is more complicated, and we shall leave it to Section \ref{sec6}.

\begin{proof}[{\bf Proof of Lemma \ref{lemapp}}]
For some locally integrable nonnegative function $f$ and large positive constant $\lam$ to be determined later, let us denote \begin{equation}\label{alem3}
v_{\lam}(t,x)
\eqdef v(t,x)\exp\Bigl(-\lam\int_0^t f(t')\,dt'\Bigr).
\end{equation}
Then we can deduce from \eqref{Saq1} that
\begin{align*}
v_{\lam}(t)=&\exp\Bigl(-\lam\int_0^t f(t')\,dt'\Bigr)e^{t\Dh}(\rho_0v_0)-\PP(\vr v_{\lam}+h_{\lam})(t)\\
&-\int_0^t\exp\Bigl(-\lam\int_{t'}^t f(\tau)\,d\tau\Bigr) e^{(t-t')\Dh}\Dh \PP(\vr v_{\lam}+h_{\lam})(t')\,dt'\\
&-\int_0^t \exp\Bigl(-\lam\int_{t'}^t f(\tau)\,d\tau\Bigr) e^{(t-t')\Dh}\PP\Bigl(\dive\bigl(\r u\otimes v_{\lam}+\r v_{\lam}\otimes u +g_{\lam}\bigr)\Bigr)(t')\,dt'.
\end{align*}
By assumptions, there holds $\f2q-s<\f2p$ and $\f2q-s>-\f2p$ (or $\f2q-s=-\f2p$ if $r=\infty$). Hence we can use Lemma \ref{lemproductlaw} to get
\begin{equation}\begin{split}\label{3.32}
\Bigl\|\exp\Bigl(-\lam\int_0^t
f(t')\,dt'\Bigr)&e^{t\Dh}(\rho_0v_0)
-\PP(\vr v_{\lam}+h_{\lam})\Bigr\|
_{\wt{L}^q_t(\dB^{\f2q-s}_{p,r})}\\
&\lesssim \|\rho_0v_0\|_{\dB^{-s}_{p,r}}
+\|\vr\|_{\wt{L}^\infty_t(\cB^{\f2p}_p)}
\|v_{\lam}\|_{\wt{L}^q_t(\dB^{\f2q-s}_{p,r})}
+\|h\|_{\wt{L}^q_t(\dB^{\f2q-s}_{p,r})}.
\end{split}\end{equation}
While for any function $\psi$, the maximal regularity estimate to heat semigroup reads
\begin{equation}\label{alem5}
\bigl\|\int_0^t e^{(t-t')\Dh}\psi(t')\,dt'
\bigr\|_{\wt{L}^{q_1}_t(\dB^{\f2{q_1}-s}_{p,r})}
\lesssim \|\psi\|_{\wt{L}^{q_2}_t(\dB^{-2+\f2{q_2}-s}_{p,r})}\,,
\quad\forall\ 1\leq q_2\leq q_1\leq\infty.
\end{equation}
By using this, we can obtain
\begin{equation}\begin{split}\label{3.34}
\bigl\|\int_0^t \exp\Bigl(-\lam\int_{t'}^t f(\tau)\,d\tau\Bigr)&
e^{(t-t')\Dh}\Dh \PP(\vr v_{\lam}+h_{\lam})(t')\,dt'
\bigr\|_{\wt{L}^q_t(\dB^{\f2q-s}_{p,r})}\\
&\lesssim\|\vr v_{\lam}+h_{\lam}\|
_{\wt{L}^q_t(\dB^{\f2q-s}_{p,r})}\\
&\lesssim \|\vr\|_{\wt{L}^\infty_t(\cB^{\f2p}_p)}
\|v_{\lam}\|_{\wt{L}^q_t(\dB^{\f2q-s}_{p,r})}
+\|h\|_{\wt{L}^q_t(\dB^{\f2q-s}_{p,r})}.
\end{split}\end{equation}

On the other hand, by applying Bony's decomposition
(see \cite{Bo81}), we write
$$\rho u\otimes v_{\lam}=T_u \rho v_{\lam}
+T_{\rho v_{\lam}}u+R(u,\rho v_{\lam}),$$
where
\begin{align*}
T_a b\eqdef\sum_{j\in\Z}S_{j-1}a\D_jb,
\andf R(a,b)\eqdef\sum_{j\in\Z}\D_ja
\widetilde{\Delta}_jb
\quad \mbox{with}\quad \widetilde{\Delta}_j\eqdef \sum_{|j'-j|\leq1}\Delta_{j'}.
\end{align*}
For any $\fq\in ]1,2[$ and $\fq'=(1-\f{1}{\fq})^{-1}\in]2,\infty[$, we can use Lemma \ref{lemBern} to obtain
$$\|S_{j-1} u\|_{L^\infty}\lesssim 2^{j(1-\f2{\fq'})}\|u\|_{\dB^{\f2{\fq'}-1}_{\infty,\infty}}
\lesssim 2^{j(1-\f2{\fq'})}\|u\|_{\dB^{\f2{\fq'}}_{2,\infty}}.$$
By using this and H\"older's inequality, for any weight $\phi(t)\geq 0$, we have
\begin{equation}\begin{split}\label{3.35}
\|\D_j(T_u \rho v_{\lam})\|_{L^1_{t,\phi}(L^p)}
\lesssim & \sum_{|j'-j|\leq 4}
\|S_{j'-1} u\|_{L^{\fq'}_{t,\phi^{\fq'}}(L^\infty)}
\|\D_{j'}(\rho v_{\lam})\|_{L^\fq_t(L^p)}\\
\lesssim &\sum_{|j'-j|\leq 4}2^{j'(1-\f2{\fq'})}
\|u\|_{L^{\fq'}_{t,\phi^{\fq'}}
(\dB^{\f2{\fq'}}_{2,\infty})}
\cdot d_{j',r}2^{j'(s-\f2{\fq})}
\|\rho v_{\lam}\|_{\wt L^\fq_t(\dB^{\f2\fq-s}_{p,r})}\\
\lesssim & d_{j,r}2^{j(s-1)}
\|u\|_{L^{\fq'}_{t,\phi^{\fq'}}(\dB^{\f2{\fq'}}_{2,\infty})}
\|\rho v_{\lam}\|_{\wt L^\fq_t(\dB^{\f2\fq-s}_{p,r})}.
\end{split}\end{equation}
Here and in all that follows, we always denote $(d_{j,r})_{k\in\Z}$ to be a generic element on the unit sphere of $\ell^r(\Z)$.
Similarly, in view of the assumption that $s+\f2p-\f2{\fq}>0$, we have
\begin{align*}
\|S_{j-1}(\rho v_{\lam})\|_{L^\fq_t(L^\infty)}
&\lesssim d_{j,r}2^{j(s+\f2p-\f2\fq)}
\|\rho v_{\lam}\|_{\wt L^\fq_t(\dB^{\f2\fq-\f2p-s}_{\infty,r})}\\
&\lesssim d_{j,r}2^{j(s+\f2p-\f2\fq)}
\|\rho v_{\lam}\|_{\wt L^\fq_t(\dB^{\f2\fq-s}_{p,r})},
\end{align*}
and thus
\begin{equation}\begin{split}
\|\D_j(T_{\rho v_{\lam}}u)\|_{L^1_{t,\phi}(L^p)}
\lesssim& 2^{j(1-\f2p)}\sum_{|j'-j|\leq 4}
\|S_{j'-1}(\rho v_{\lam})\|_{L^\fq_t(L^\infty)}
\|\D_{j'} u\|_{L^{\fq'}_{t,\phi^{\fq'}}(L^2)}\\
\lesssim & d_{j,r}2^{j(s-1)}
\|u\|_{L^{\fq'}_{t,\phi^{\fq'}}(\dB^{\f2{\fq'}}_{2,\infty})}
\|\rho v_{\lam}\|_{\wt L^\fq_t(\dB^{\f2\fq-s}_{p,r})}.
\end{split}\end{equation}
For the remainder term, it follows from Lemma \ref{lemBern} that for any $s<2$, we have
\begin{equation}\begin{split}\label{3.37}
\|\D_j R(u,\rho v_{\lam})\|_{L^1_{t,\phi}(L^p)}
\lesssim & 2^j\sum_{j'\geq j-5}
\|\D_{j'}u\|_{L^{\fq'}_{t,\phi^{\fq'}}(L^2)}
\|\wt{\D}_{j'}(\rho v_{\lam})\|_{L^\fq_t(L^p)}\\
\lesssim & 2^j\sum_{j'\geq j-5} d_{j',r}2^{j'(s-2)}
\|u\|_{L^{\fq'}_{t,\phi^{\fq'}}(\dB^{\f2{\fq'}}_{2,\infty})}
\|\rho v_{\lam}\|_{\wt L^\fq_t(\dB^{\f2\fq-s}_{p,r})}\\
\lesssim & d_{j,r}2^{j(s-1)}
\|u\|_{L^{\fq'}_{t,\phi^{\fq'}}
(\dB^{\f2{\fq'}}_{2,\infty})}
\|\rho v_{\lam}\|
_{\wt L^\fq_t(\dB^{\f2\fq-s}_{p,r})}.
\end{split}\end{equation}

By combining the above three estimates \eqref{3.35}-\eqref{3.37}, we deduce
\begin{align*}
\|\r u\otimes v_{\lam}+\r v_{\lam}\otimes u\|
_{\wt L^1_{t,\phi}(\dB^{1-s}_{p,r})}
\lesssim\|u\|_{L^{\fq'}_{t,\phi^{\fq'}}
(\dB^{\f2{\fq'}}_{2,\infty})}
\|\rho\|_{\wt L^\infty_t(\cB^{\f2p}_p)}
\|v_{\lam}\|_{\wt L^\fq_t(\dB^{\f2\fq-s}_{p,r})},
\quad \forall\ s\in ]1-2/p,2[\,,
\end{align*}
which together with \eqref{alem5} ensures that
\begin{align*}
\cA\eqdef&\Bigl\|\int_0^t \exp\Bigl(-\lam\int_{t'}^t f(\tau)\,d\tau\Bigr)
e^{(t-t')\D}\PP\diveh(\r u\otimes v_{\lam}
+\r v_{\lam}\otimes u+g_{\lam})(t')\,dt'
\Bigr\|_{\wt{L}^q_t(\dB^{\f2q-s}_{p,r})}\\
&\lesssim\Bigl( \int_0^t \exp\Bigl(-\lam \fq'\int_{t'}^t f(\tau)\,d\tau\Bigr)\| u(t')\|_{\dB^{\f2{\fq'}}_{2,\infty}}^{\fq'}
\,dt'\Bigr)^{\f1{\fq'}}
\|v_{\lam}\|_{\wt L^\fq_t(\dB^{\f2\fq-s}_{p,r})}
+\|g_{\lam}\|_{\wt L^1_t(\dB^{1-s}_{p,r})}.
\end{align*}

Now let us take $f(t)=\|u(t)\|_{\dB_{2,\infty}^{\f2{\fq'}}}^{\fq'}$,
which is reasonable since
\begin{equation}\label{controlfr}
\int_0^t \|u(t')\|_{\dB_{2,\infty}^{\f2{\fq'}}}^{\fq'}\,dt'
\leq\int_0^t \|u(t')\|_{L^2}^{\fq'-2}
\|\nablah u(t')\|_{L^2}^2\,dt'
\leq C\|u_0\|_{L^2}^{\fq'},\quad\forall\ t>0.
\end{equation}
For this choice of $f$, we can directly calculate that
$$\int_0^t \exp\Bigl(-\lam \fq'\int_{t'}^t f(\tau)\,d\tau\Bigr)
\| u(t')\|_{\dB^{\f2{\fq'}}_{2,\infty}}^{\fq'}\,dt'
=\f{1}{\lam \fq'}\Bigl(1-\exp(-\lam \fq'\int_{0}^t f(\tau)\,d\tau)\Bigr)
\leq\f{1}{\lam \fq'}.$$
As a result, we obtain
\begin{equation}\label{3.39}
\cA\lesssim\lam^{-\f1{\fq'}}
\|v_{\lam}\|_{\wt L^\fq_t(\dB^{\f2\fq-s}_{p,r})}
+\|g\|_{\wt L^1_t(\dB^{1-s}_{p,r})}.
\end{equation}

Now by combining the estimates \eqref{3.32},~\eqref{3.34} and \eqref{3.39}, we finally achieve
\begin{equation}\begin{split}\label{alem6}
\|v_{\lam}\|_{\wt{L}^q_t(\dB^{\f2q-s}_{p,r})}
\leq C\Bigl(\|\rho_0v_0\|_{\dB^{-s}_{p,r}}
&+\|\varrho\|_{\wt{L}^\infty_t(\cB^{\f2p}_p)}
\|v_{\lam}\|_{\wt{L}^q_t(\dB^{\f2q-s}_{p,r})}
+\lam^{-\f1{\fq'}}\|v_{\lam}\|_{\wt L^\fq_t(\dB^{\f2\fq-s}_{p,r})}\\
+&\|h\|_{\wt{L}^q_t(\dB^{\f2q-s}_{p,r})}
+\|g\|_{\wt L^1_t(\dB^{1-s}_{p,r})}\Bigr),
\quad \forall\ \s\in ]1-2/p,2[.
\end{split}\end{equation}
In particular, if we take $q=\fq$ (which is always permitted) in \eqref{alem6}, $\lam$ large enough and $\|\varrho\|_{\wt{L}^\infty_t(\cB^{2/p}_p)}$ small enough so that $C(\|\varrho\|_{\wt{L}^\infty_t(\cB^{2/p}_p)}
+\lam^{-\f1{\fq'}})\leq \f12$,
then \eqref{alem6} gives
$$\|v_{\lam}\|_{\wt{L}^\fq_t(\dB^{\f2\fq-s}_{p,r})}
\leq C\Bigl(\|\rho_0v_0\|_{\dB^{-s}_{p,r}}
+\|h\|_{\wt{L}^\fq_t(\dB^{\f2\fq-s}_{p,r})}
+\|g\|_{\wt L^1_t(\dB^{1-s}_{p,r})}\Bigr).$$
By inserting this into \eqref{alem6}, we achieve for any $s\in ]1-2/p,2[$ that
\begin{equation}\label{alem7}
\|v_{\lam}\|_{\wt{L}^q_t(\dB^{\f2q-s}_{p,r})}
\leq C\Bigl(\|\rho_0v_0\|_{\dB^{-s}_{p,r}}
+\|h\|_{\wt{L}^q_t(\dB^{\f2q-s}_{p,r})
\cap\wt{L}^\fq_t(\dB^{\f2\fq-s}_{p,r})}
+\|g\|_{\wt L^1_t(\dB^{1-s}_{p,r})}\Bigr).
\end{equation}
Yet it follows from the bound \eqref{controlfr} that
\begin{align*}
\|v\|_{\wt{L}^q_t(\dB^{\f2q-s}_{p,r})}
&\leq\|v_{\lam}\|_{\wt{L}^q_t(\dB^{\f2q-s}_{p,r})}
\exp\Bigl(\lam\int_0^t \|u(t')\|_{\dB_{2,\infty}^{\f2{\fq'}}}^{\fq'}\,dt'\Bigr)\\
&\leq\|v_{\lam}\|_{\wt{L}^q_t(\dB^{\f2q-s}_{p,r})}
\exp\bigl(C\lam\|u_0\|_{L^2}^{\fq'}\bigr),
\end{align*}
which together with \eqref{alem7} leads to \eqref{Saq2}.
\smallskip

While for the case when $s=2$, we can use \eqref{3.32} and \eqref{3.34} with $f=0$ to get
\begin{equation}\begin{split}\label{3.42}
\|v\|_{\wt{L}^q_t(\dB^{\f2q-2}_{p,\infty})}
\leq C\bigl(\|\rho_0v_0\|_{\dB^{-2}_{p,\infty}}
+\|\varrho\|_{\wt{L}^\infty_t(\cB^{\f2p}_p)}
\|v\|_{\wt{L}^q_t(\dB^{\f2q-2}_{p,\infty})}
+\|h\|_{\wt{L}^q_t(\dB^{\f2q-2}_{p,\infty})}
+\cB\bigr),
\end{split}\end{equation}
where
$$\cB\eqdef\Bigl\|\int_0^t
e^{(t-t')\D}\PP\diveh(\r u\otimes v
+\r v\otimes u+g)(t')\,dt'
\Bigr\|_{\wt{L}^q_t(\dB^{\f2q-2}_{p,\infty})}.$$
By using \eqref{alem5}, Minkowski's inequality and the relation $L^{\f{2p}{2+p}}\hookrightarrow \dB^{-1}_{p,\infty}$, we have
\begin{align*}
\cB&\lesssim\|\rho u\otimes v+\rho v\otimes u\|_{\wt L^1_t(\dB^{-1}_{p,\infty})}
+\|g\|_{\wt L^1_t(\dB^{-1}_{p,\infty})}\\
&\leq\|\rho u\otimes v+\rho v\otimes u\|_{L^1_t(\dB^{-1}_{p,\infty})}
+\|g\|_{\wt L^1_t(\dB^{-1}_{p,\infty})}\\
&\lesssim\|\rho\|_{L^\infty_t(L^\infty)}
\|u\|_{L^2_t(L^2)}\|v\|_{L^2_t(L^p)}
+\|g\|_{\wt L^1_t(\dB^{-1}_{p,\infty})}.
\end{align*}
By substituting this into \eqref{3.42} and taking $\|\varrho\|_{\wt{L}^\infty_t(\cB^{2/p}_p)}$ to be small enough, we achieve the desired estimate \eqref{endpoint}.
This completes the proof of Lemma \ref{lemapp}.
\end{proof}

\section{Decay estimates of $\|u(t)\|_{\cB^0_p}$ when $\s\in]0,2[$}\label{sec4}

The aim of this section is to give decay estimates of $\|u(t)\|_{\cB^0_p}$ when $\s\in]0,2[$, namely
\begin{equation}\label{4.1}
\|u(t)\|_{\cB^0_p}
\lesssim \alp^2\w{t}^{\f1p-\f12-\f\s2},
\quad\forall\ \s\in]0,2[.
\end{equation}
We mention that the case when $p=2$ has already been given by \eqref{3.12}. Thus in the rest of this section, we always assume that $p>2$.

{\bf Step 1.}
Let us first consider the case when $\s\in ]0,1+2/p[\,$. In this case, we always have $2(1+\s-2/p\pm\tau)^{-1}\in\Omega_\s$ whenever $0<\tau<\min\{1+2/p-\s,2/p\}$. Then by applying \eqref{SAeq1} with $q=2(1+\s-2/p\pm\tau)^{-1}$ gives
\begin{align*}
\|u\|_{L^{\f2{1+\s-\f2p}}_t(\cB^{1-\f2p}_{2})}
\lesssim\|u\|_{\wt L^{\f2{1+\s-\f2p+\tau}}_t
(\dB^{1-\f2p+\tau}_{2,\infty})}^{\f12}
\|u\|_{\wt L^{\f2{1+\s-\f2p-\tau}}_t
(\dB^{1-\f2p-\tau}_{2,\infty})}^{\f12}
\lesssim\|\rho_0 u_0\|_{\dB^{-\s}_{2,\infty}}
\exp\bigl(C\|u_0\|_{L^2}^4\bigr),
\end{align*}
which ensures the existence of
some $t_2\in ]0,t[$ so that
\begin{equation}\label{4.3}
\|u(t_2)\|_{\cB^0_{p}}
\lesssim\|u(t_2)\|_{\cB^{1-\f2p}_{2}}
\lesssim\|\rho_0 u_0\|_{\dB^{-\s}_{2,\infty}}
\exp\bigl(C\|u_0\|_{L^2}^4\bigr)t^{\f1p-\f12-\f\s2}.
\end{equation}
On the other hand, for any $p\in]2,\infty[$ and $s\in ]-2/p,1/p]$, \eqref{lem4.1b} ensures
\begin{equation}\begin{split}\label{a4.4}
\|u\|_{\wt{L}^\infty(\tau,t;\cB^s_{p})}
+\|u\|_{L^1(\tau,t;\cB^{s+2}_{p})}
\leq 2\|u(\tau)\|_{\cB^s_{p}}
\exp\bigl(C\|u_0\|_{L^2}^4\bigr),
\quad\forall\ 0\leq\tau<t.
\end{split}\end{equation}
By using \eqref{a4.4} with $s=0$ and then \eqref{4.3},
we get for any $\s\in ]0,1+2/p[$ that
\begin{equation}\begin{split}\label{a5.1}
\|u(t)\|_{\cB^0_{p}}
\leq 2\min\bigl\{\|u_0\|_{\cB^0_{p}},
\|u(t_2)\|_{\cB^0_{p}}\bigr\}
\exp\bigl(C\|u_0\|_{L^2}^4\bigr)
\lesssim\alp \w{t}^{\f1p-\f12-\f\sigma2}.
\end{split}\end{equation}

In particular, when $\sigma=2/3+2/p$, \eqref{a5.1} gives
\begin{equation}\label{uBp01}
\|u(t)\|_{\cB^0_{p}}
\lesssim\al_{\f23+\f2p,p} \w{t}^{-\f56}.
\end{equation}

{\bf Step 2.}
Next, when $1+2/p\leq\sigma<2$,
the assumptions that $\rho_0u_0\in \dB^{-\sigma}_{2,\infty}\cap\cB^0_2$ actually imply $\rho_0u_0\in \dB^{-(\f23+\f2p)}_{2,\infty}$, and $\al_{\f23+\f2p,p}\lesssim\alp$. As a result, we can use \eqref{uBp01} to get
\begin{equation}\label{4.6}
\|u(t)\|_{\cB^0_{p}}
\lesssim\al_{\s,p} \w{t}^{-\f56},
\quad\forall\ 1+2/p\leq\sigma<2,
\end{equation}
which is a first-step estimate. In the following, we shall refine it to fit \eqref{4.1}. To do this, we use the integral form \eqref{alem2} again to obtain
\begin{equation}\begin{split}\label{a5.1a}
\|u(t)\|_{\cB^0_p}\leq &\|e^{t\Dh}(\rho_0u_0)\|_{\cB^0_p}
+\|\PP(\vr u)(t)\|_{\cB^0_p}+\bigl\|\int_0^te^{(t-t')\Dh}\Dh \PP(\vr u)(t')\,dt'\bigr\|_{\cB^0_p}\\
&+\bigl\|\int_0^te^{(t-t')\Dh}\PP
\bigl(\diveh(\r u\otimes u)\bigr)(t')\,dt'\bigr\|_{\cB^0_p}.
\end{split}\end{equation}

We first get by using Lemma \ref{lemproductlaw} and the fact that $\|\vr\|_{\wt{L}^\infty_t(\cB^1_{2})}$ is sufficiently small that
\begin{equation}\label{a8}
\|\PP(\vr u)(t)\|_{\cB^0_p}
\leq C\|\vr(t)\|_{\cB^1_{2}}
\|u(t)\|_{\cB^0_p}
\leq\f12\|u(t)\|_{\cB^0_p}.
\end{equation}

While for any $p_1\geq p_2$ so that $\f{s_1-s_2}2+\f1{p_2}-\f1{p_1}>0$, there holds
\begin{equation}\begin{split}\label{a6}
\|e^{t\Dh} f\|_{\cB^{s_1}_{p_1}}
&\lesssim\sum_{j\in\Z} 2^{js_1}e^{-Ct2^{2j}}
\|\dhk f\|_{L^{p_1}}\\
&\lesssim t^{-(\f{s_1-s_2}2+\f1{p_2}-\f1{p_1})}
\sum_{j\in\Z} \bigl(t2^{2j}\bigr)^{\f{s_1-s_2}2+\f1{p_2}-\f1{p_1}}
e^{-Ct2^{2j}}\|f\|_{\dB^{s_2}_{p_2,\infty}}\\
&\lesssim t^{-(\f{s_1-s_2}2+\f1{p_2}-\f1{p_1})}
\|f\|_{\dB^{s_2}_{p_2,\infty}},\quad\forall\ t>0.
\end{split}\end{equation}
In particular, \eqref{a6} implies that
\begin{equation}\label{a7}
\|e^{t\Dh}(\rho_0u_0)\|_{\cB^0_p}\lesssim  t^{\f1p-\f12-\f\sigma2}
\|\rho_0u_0\|_{\dB^{-\sigma}_{2,\infty}}.
\end{equation}

Thanks to \eqref{a6} again, we get by using \eqref{3.12} and $\s<2$ that
\begin{align*}
\bigl\|\int_0^{\f{t}2}e^{(t-t')\Dh}\Dh\PP(\vr u )(t')\,dt'\bigr\|_{\cB^0_p}
&\lesssim\int_0^{\f{t}2} (t-t')^{-(\f32-\f1p)}
\|\vr(t')\|_{\cB^1_{2}}\|u(t')\|_{\cB^0_{2}}\,dt'\\
&\lesssim t^{-(\f32-\f1p)}
\int_0^{\f{t}2}\alp\langle t'\rangle^{-\f\s2}\,dt'\\
&\lesssim \alp t^{\f1p-\f12-\f\sigma2},
\end{align*}
and by using \eqref{a4.4} with $s=0$ that
\begin{align*}
\bigl\|\int_{\f{t}2}^t e^{(t-t')\Dh}\Dh \PP(\vr u)(t')\,dt'\bigr\|_{\cB^0_p}
&\lesssim\sum_{j\in\Z}\int_{\f{t}2}^t e^{-C(t-t')2^{2j}}2^{2j}
\|\dhk(\vr u)(t')\|_{L^p}\,dt'\\
&\lesssim \|\vr u\|
_{\wt{L}^\infty(\f t2,t;\cB^0_p)}
\sum_{j\in\Z}d_{j,1}
2^{2j}\int_{\f{t}2}^t e^{-C(t-t')2^{2j}}\,dt'\\
&\lesssim\|\vr\|_{\wt{L}^\infty(\f t2,t;\cB^1_{2})}
\|u\|_{\wt{L}^\infty(\f t2,t;B^0_p)}\\
&\lesssim\|\vr\|_{\wt{L}^\infty_t(\cB^1_{2})}
\|u(t/2)\|_{\cB^0_p}\exp\bigl(C\|u_0\|_{L^2}^4\bigr).
\end{align*}
Combining the above two inequalities gives
\begin{equation}
\bigl\|\int_{0}^t e^{(t-t')\Dh}\Dh \PP(\vr v)(t')\,dt'\bigr\|_{\cB^0_p}
\lesssim\|\vr\|_{\wt{L}^\infty_t(\cB^1_{2})}
\exp\bigl(C\|u_0\|_{L^2}^4\bigr)
\|u(t/2)\|_{\cB^0_p}
+\alp t^{\f1p-\f12-\f\sigma2}.
\end{equation}

Whereas due to $\sigma\geq1+2/p$, we deduce from \eqref{a6} and then \eqref{3.12} that
\begin{align*}
\int_0^{\f t2}\bigl\|e^{(t-t')\Dh}
\PP\diveh(\r u\otimes u)(t')
\bigr\|_{\cB^0_p}\,dt'
&\lesssim\int_0^{\f t2}(t-t')^{-(\f32-\f1p)}
\|\diveh(\r u\otimes u)(t')\|_{\dB^{-1}_{1,\infty}}\,dt'\\
&\lesssim t^{\f1p-\f32}\int_0^{\f t2}
\|\rho(t')\|_{L^\infty}\|u(t')\|_{L^2}\|u(t')\|_{L^2}
\,dt'\\
&\lesssim\alp^2 t^{\f1p-\f32}\int_0^{\f t2}
\langle t' \rangle^{-\sigma}\,dt'\\
&\lesssim\alp^2 t^{\f1p-\f32},
\end{align*}
and by using \eqref{3.12},~\eqref{4.6} and the choice of $p>2$ so that
$-(\f9{10}+\f1{5p})>-1$ that
\begin{align*}
\int_{\f t2}^t\bigl\| e^{(t-t')\Dh}
\PP\diveh(\r u\otimes u)(t')\bigr\|_{\cB^0_p}\,dt'
&\lesssim\int_{\f t2}^t(t-t')^{-(\f9{10}+\f1{5p})}
\|\diveh(\r u\otimes u)(t')\|_{\dB^{-1}_{\f{5p}{2p+6},\infty}}\,dt'\\
&\lesssim\int_{\f t2}^t
(t-t')^{-(\f9{10}+\f1{5p})} \|u(t')\|_{L^2}^{\f45}
\|u(t')\|_{L^p}^{\f65}\,dt'\\
&\lesssim \alp^2
\int_{\f t2}^t (t-t')^{-(\f9{10}+\f1{5p})}
\langle t'\rangle^{-(\f25\s+1)}\,dt'\\
&\lesssim \alp^2
t^{-(\f25\s+\f{9}{10}+\f{1}{5p})}.
\end{align*}
Combining the above two inequalities together with the facts that
$$\f1p-\f32<\f1p-\f12-\f\sigma2,\andf -\bigl(\f25\s+\f{9}{10}+\f{1}{5p}\bigr)
<\f1p-\f12-\f\sigma2$$
leads to
\begin{equation}\label{a10}
\int_0^t\bigl\| e^{(t-t')\Dh}
\PP\diveh(\r u\otimes u)(t')\bigr\|_{\cB^0_p}\,dt'
\lesssim \alp^2
t^{\f1p-\f12-\f\sigma2},\quad\forall\ t\geq1.
\end{equation}

Now by inserting the estimates \eqref{a8} and \eqref{a7}-\eqref{a10} into \eqref{a5.1a}, we achieve
\begin{equation}\label{a91}
\|u(t)\|_{\cB^0_p}
\lesssim \|\vr\|_{\wt{L}^\infty_t(\cB^1_{2})}
\exp\bigl(C\|u_0\|_{L^2}^4\bigr)
\|u(t/2)\|_{\cB^0_p}
+\alp^2 t^{\f1p-\f12-\f\sigma2},\quad\forall\ t\geq1.
\end{equation}
While \eqref{4.6} in particular gives
\begin{equation}\label{a92}
\|u(t)\|_{\cB^0_p}\leq \alp,\quad\forall\ t\leq1.
\end{equation}
In view of \eqref{a91},~\eqref{a92} and the fact that $\|\vr\|_{\wt{L}^\infty_t(\cB^1_{2})}
\exp(C\|u_0\|_{L^2}^4)$ is sufficiently small, then we can use Lemma \ref{lemit} to achieve
$$\|u(t)\|_{\cB^0_p}
\lesssim \alp^2 t^{\f1p-\f12-\f\sigma2},
\quad\forall\ t\geq1,$$
which together with \eqref{a92} ensures the desired estimate \eqref{4.1} when $1+2/p\leq\sigma<2$.

\section{Decay estimates of $\|u(t)\|_{\cB^2_p}$ when $\s\in]0,2[$}\label{sec5}
In view of \eqref{4.1} and \eqref{a4.4}, there holds
$$\|u\|_{\wt{L}^\infty(t/2,t;\cB^0_{p})}
+\|u\|_{L^1(t/2,t;\cB^{2}_{p})}
\leq 2\|u(t/2)\|_{\cB^0_{p}}
\exp\bigl(C\|u_0\|_{L^2}^4\bigr)
\lesssim\alp^2 \langle t\rangle^{\f1p-\f12-\f\sigma2}.$$
As a result, there exists some $t_3\in]t/2,t[$ such that
$$\|u(t_3)\|_{\cB^2_{p}}
\lesssim\alp^2 \langle t\rangle^{\f1p-\f12-\f\sigma2}t^{-1}.$$
Then the interpolation inequality gives for any $0\leq\th\leq2$ that
\begin{equation}\label{a11}
\|u(t_3)\|_{\cB^\th_{p}}\leq\|u(t_3)\|_{\cB^0_{p}}^{1-\f{\th}2}
\|u(t_3)\|_{\cB^2_{p}}^{\f{\th}2}
\lesssim\alp^2 \langle t\rangle^{\f1p-\f12-\f\sigma2}t^{-\f{\th}2}.
\end{equation}
Then for any $\th_1\in[0,2/p]$, we can get by applying Lemma \ref{lem4.1} and \eqref{a2} that
\begin{equation}\begin{split}\label{a12}
\|u\|_{\wt{L}^\infty(t_3,t;\cB^{\th_1}_{p})}
+\|\Dh u\|_{L^1(t_3,t;\cB^{\th_1}_{p})}
&\leq 2\|u(t_3)\|_{\cB^{\th_1}_{p}}
\exp\bigl(\exp(C\|u_0\|_{\cB^0_{2}}^4)\bigr)\\
&\lesssim\alp^2 \langle t\rangle^{\f1p-\f12-\f\sigma2}t^{-\f{\th_1}2}.
\end{split}\end{equation}

On the other hand, by applying $\pa_t$ to the equation of $u$ in \eqref{2dua}, we get
\begin{align*}
u_{tt}+u\cdot\nablah u_t+u_t\cdot\nablah u-(1+a)\D u_t
=-(1+a)\nablah P_t+a_t(\Dh u-\nablah P).
\end{align*}
Applying Lemma \ref{lem4.1} with $b_1=b_2=u$
and $\eta=a_t(\Dh u-\nabla P)$ leads to
\begin{equation}\begin{split}\label{a13}
&\|u_t\|_{\wt{L}^\infty(t_3,t;\cB^{0}_{p})}
+\|\Dh u_t\|_{L^1(t_3,t;\cB^{0}_{p})}
\leq \exp\bigl(\exp(C\|u_0\|_{\cB^0_{2}}^4)\bigr)\\
&\qquad\times\Bigl(2\|u_t(t_3)\|_{\cB^{0}_{p}}
+C\|a_t \|_{\wt{L}^\infty(t_3,t;\cB^0_{p})}
\bigl(\|\Dh u\|_{L^1(t_3,t;\cB^{\f2p}_{p})}
+\|\nablah P\|_{L^1(t_3,t;\cB^{\f2p}_{p})}\bigr)\Bigr).
\end{split}\end{equation}
By taking divergence to the equation of $u$ in \eqref{2dua}, we get
\begin{equation}\label{a13p} -\diveh\bigl((1+a)\nablah P\bigr)
=\diveh\bigl(u\cdot\nablah u-a\Dh u\bigr),
\end{equation}
from which,  Lemma \ref{lemp} and \eqref{a12}, we infer
\begin{align*}
\|\Dh u\|_{L^1(t_3,t;\cB^{\f2p}_{p})}
+\|\nablah P\|_{L^1(t_3,t;\cB^{\f2p}_{p})}
&\lesssim\|\Dh u\|_{L^1(t_3,t;\cB^{\f2p}_{p})}
+\|u\cdot\nablah u\|_{L^1(t_3,t;\cB^{\f2p}_{p})}\\
&\lesssim\|\Dh u\|_{L^1(t_3,t;\cB^{\f2p}_{p})}
+(t-t_3)^{\f12}\|u\|_{L^\infty(t_3,t;\cB^{\f2p}_{p})}
\|\nablah u\|_{L^2(t_3,t;\cB^{\f2p}_{p})}\\
&\lesssim\alp^2 \langle t\rangle^{\f1p-\f12-\f\sigma2}t^{-\f1p}
\Bigl(1+(t-t_3)^{\f12}\alp^2\langle t\rangle^{\f1p-\f12-\f\sigma2}t^{-\f1p}\Bigr).
\end{align*}
Yet due to $t_3\in]t/2,t[$ and $p\geq2,\,\s>0$, one has
$$(t-t_3)^{\f12}\langle t\rangle^{\f1p-\f12-\f\sigma2}t^{-\f1p}
\leq \langle t\rangle^{\f1p-\f12-\f\sigma2}t^{\f12-\f1p}
\lesssim 1.$$
As a result, we achieve
\begin{equation}\label{a14}
\|\Dh u\|_{L^1(t_3,t;\cB^{\f2p}_{p})}
+\|\nablah P\|_{L^1(t_3,t;\cB^{\f2p}_{p})}
\lesssim\alp^4 \langle t\rangle^{\f1p-\f12-\f\sigma2}t^{-\f1p}.
\end{equation}

On the other hand, by using Theorem $3.14$ of \cite{BCD} and the estimate \eqref{a2}, we have
\begin{align*}
\|a\|_{\wt{L}^\infty_t(\cB^{2-\f2p}_{2})}
\leq\|a_0\|_{\cB^{2-\f2p}_{2}}
\exp\bigl(C\|\nablah u\|_{L^1_t(\cB^1_{2})}\bigr)
\leq\|a_0\|_{\cB^{2-\f2p}_{2}}
\exp\bigl(C\|u_0\|_{\cB^0_{2}}
\exp(C\|u_0\|_{L^2}^4)\bigr),
\end{align*}
which together with \eqref{a2} guarantees that
\begin{equation}\label{a2vr}
\|a\|_{\wt{L}^\infty_t(\cB^{\th_2}_{2})}
\leq C,\quad\forall\ \th_2\in[1,{2-\f2p}].
\end{equation}
In view of the equation for $a$ in \eqref{2dua},
we deduce from \eqref{a12} and \eqref{a2vr} that
\begin{equation}\begin{split}\label{a15}
\|a_t\|_{\wt{L}^\infty(t_3,t;\cB^0_{p})}
\lesssim\|\nablah a\|_{\wt{L}^\infty(t_3,t;\cB^0_{p})}
\|u\|_{\wt{L}^\infty(t_3,t;\cB^{\f2p}_{p})}
\lesssim \alp^2\langle t\rangle^{\f1p-\f12-\f\sigma2}t^{-\f1p}.
\end{split}\end{equation}
And in view of the equation for $u$ in \eqref{2dua}, we deduce from  Lemma \ref{lemp}
and \eqref{a11} that
\begin{equation}\begin{split}\label{a16}
\|u_t(t_3)\|_{\cB^{0}_{p}}&\lesssim\|(1+a)\Dh u(t_3)\|_{\cB^{0}_{p}}
+\|(u\cdot\nablah u)(t_3)\|_{\cB^{0}_{p}}
+\|(1+a)\nablah P(t_3)\|_{\cB^{0}_{p}}\\
&\lesssim\|\Dh u(t_3)\|_{\cB^{0}_{p}}
+\|u(t_3)\|_{\cB^{\f2p}_{p}}\|\nablah u(t_3)\|_{\cB^{0}_{p}}\\
&\lesssim\alp^2 \langle t\rangle^{\f1p-\f12-\f\sigma2}t^{-1}
+\alp^4 \langle t\rangle^{\f2p-1-\sigma}t^{-\f12-\f1p}\\
&\lesssim\alp^4 \langle t\rangle^{\f1p-\f12-\f\sigma2}t^{-1}.
\end{split}\end{equation}

By substituting the estimates \eqref{a14}-\eqref{a16}
into \eqref{a13}, we conclude that
$$\|u_t\|_{\wt{L}^\infty(t_3,t;\cB^{0}_{p})}
+\|\Dh u_t\|_{L^1(t_3,t;\cB^{0}_{p})}
\leq\alp^6 \langle t\rangle^{\max\{\f1p-\f12-\f\sigma2,-\s\}}t^{-1},$$
from which and \eqref{a12}, we deuce from the equation of $u$ in \eqref{1.1} that
\begin{align*}
\|\Dh u(t)\|_{\cB^{0}_{p}}&\leq C\bigl(\|u_t(t)\|_{\cB^{0}_{p}}
+\|u(t)\|_{\cB^{\f2p}_{p}}\|\nablah u(t)\|_{\cB^{0}_{p}}\bigr)\\
&\leq C\alp^6 \langle t\rangle^{\max\{\f1p-\f12-\f\sigma2,-\s\}}t^{-1}
+\f12\|\Dh u(t)\|_{\cB^{0}_{p}}
+C\|u(t)\|_{\cB^{\f2p}_{p}}^2\|u(t)\|_{\cB^{0}_{p}}\\
&\leq C\alp^6 \langle t\rangle^{\max\{\f1p-\f12-\f\sigma2,-\s\}}t^{-1}
+\f12\|\Dh u(t)\|_{\cB^{0}_{p}},
\end{align*}
which implies
\begin{equation}\label{a18}
\|u(t)\|_{\cB^2_{p}}\lesssim
\alp^6 \langle t\rangle^{\max\{\f1p-\f12-\f\sigma2,-\s\}}t^{-1},
\quad\forall\ \s\in]0,2[.
\end{equation}

\section{The proof of Theorem \ref{thm1}}\label{sec6}

Now by interpolation between \eqref{4.1} and \eqref{a18}, we conclude
\begin{equation}\label{6.1}
\|u(t)\|_{\cB^{\th}_{p}}\lesssim\alp^6
\langle t\rangle^{\max\{\f1p-\f12-\f\sigma2,-\s\}}
t^{-\f{\th}2},\quad\forall\ p\in[2,\infty[,~\th\in [0,2],~\s\in]0,2[,
\end{equation}
which together with \eqref{3.12} completes the proof of Theorem \ref{thm1} when $\s\in]0,2[$.
\smallskip

While the case for $\s=2$ needs more effort. We first deduce from the assumption $\rho_0u_0\in \dB^{-2}_{2,\infty}\cap\cB^0_2$ that $\rho_0u_0\in\cB^{-2+\varsigma}_2$ for any $\varsigma\in]0,1[$. Then we can use \eqref{6.1} with $\s=2-\varsigma$ to get that in this case, there holds
\begin{equation}\label{6.2}
\|u(t)\|_{\cB^{\th}_2}\lesssim\atp^6
\langle t\rangle^{-1+\f\varsigma2}
t^{-\f{\th}2},\quad\forall\ \th\in [0,2].
\end{equation}
Substituting this into \eqref{3.11} gives
for any $q\in[1,2]$ that
$$\|u\|_{\wt L^q_t(\dB^{\f2q-2}_{2,\infty})}
\lesssim\atp^{12},~\text{ i.e. }~
\|\dj u\|_{L^q_t(L^2)}
\lesssim2^{j(2-\f2q)}\atp^{12},
\quad\forall\ j\in\Z,$$
which guarantees the existence of some $t_j\in]t/4,t/2[$ such that
\begin{equation}\label{6.3}
\|\dj u(t_j)\|_{L^2}
\lesssim2^{j(2-\f2q)}\atp^{12} t^{-\f1q},
\quad\forall\ j\in\Z.
\end{equation}
Then we can get by using the same procedure of deriving \eqref{eqtDj} that
\begin{equation}\label{6.4}
\|\dj u(t/2)\|_{L^2}
+\int_{t_j}^{\f t2} 2^{2j}\|\dj u(t')\|_{L^2}\,dt'
\lesssim \|\dj u(t_j)\|_{L^2}+\cA_j,
\end{equation}
where
$$\cA_j\eqdef\int_{t_j}^{\f t2}\bigl\|\dj\bigl(a\D u
-(1+a)\na P-u\cdot\nabla u\bigr)(t')\bigr\|_{L^2}\,dt'.$$
While by using Lemma \ref{lemp} and the estimates \eqref{a2vr} and \eqref{6.2}
\begin{align*}
2^{j(\f2q-2)}\cA_j&\lesssim\int_{t_j}^{\f t2}\|a\D u(t')\|_{\cB^{\f2q-2}_2}
+\|u\cdot\nabla u(t')\|_{\cB^{\f2q-2}_2}\,dt'\\
&\lesssim\int_{t_j}^{\f t2}\|a(t')\|_{\cB^{1-\varsigma}_2}
\|\D u(t')\|_{\cB^{\f2q-2+\varsigma}_2}
+\|u(t')\|_{\cB^1_2}
\|\nabla u(t')\|_{\cB^{\f2q-2}_2}\,dt'\\
&\lesssim\atp^{12}t^{-\f1q},\quad\forall\
q\in]1,2],~\varsigma\in]0,2-2/q].
\end{align*}
By substituting this estimate and \eqref{6.3} into \eqref{6.4}, we deduce
for any $q\in]1,2]$ that
$$2^{j(\f2q-2)}\|\dj u(t/2)\|_{L^2}
\lesssim \atp^{12}t^{-\f1q},\quad\forall\ j\in\Z,\quad\text{i.e.}\quad
\|u(t/2)\|_{\dB^{\f2q-2}_{2,\infty}}
\lesssim \atp^{12}t^{-\f1q}.$$
Then the interpolation inequality gives
$$\|u(t/2)\|_{\cB^{-\f14}_{2}}
\lesssim\|u(t/2)\|_{\dB^{-\f18}_{2,\infty}}^{\f12}
\|u(t/2)\|_{\dB^{-\f38}_{2,\infty}}^{\f12}
\lesssim \atp^{12}t^{-\f78}.$$
By using this and \eqref{lem4.1b} in  Lemma \ref{lem4.1}, we deduce
\begin{align*}
\|u\|_{\wt{L}^\infty(t/2,t;\cB^{-\f14}_2)}
+\|u\|_{L^1(t/2,t;\cB^{\f74}_2)}
\lesssim\|u(t/2)\|_{\cB^{-\f14}_{2}}
\exp\bigl(C\|u_0\|_{L^2}^4\bigr)
\lesssim \atp^{12}t^{-\f78}.
\end{align*}
In particular, this gives
$$\|u\|_{L^8(t/2,t;\cB^0_2)}
\leq\|u\|_{\wt{L}^\infty(t/2,t;\cB^{-\f14}_2)}^{\f78}
\|u\|_{L^1(t/2,t;\cB^{\f74}_2)}^{\f18}
\lesssim \atp^{12}t^{-\f78}.$$
Then we can use Lemma \ref{Salem3}
together with \eqref{a2} to achieve
\begin{equation}\label{6.5}
\|u(t)\|_{\cB^0_2}\lesssim \atp^{12}\w{t}^{-1}.
\end{equation}

Now starting with \eqref{6.5}, we can repeat the procedures in Sections
\ref{sec4} and \ref{sec5} to get the desired estimate
\eqref{1.7} for $\s=2$.
This completes the proof of Theorem \ref{thm1}.

\appendix

\section{The proof of Lemma \ref{lem4.1}}\label{appA}

The aim of this appendix is to present the proof of Lemma \ref{lem4.1}.
Before proceeding, let us give the following lemma
that can help us to handle the pressure term.

\begin{lem}\label{lemp}
{\sl Let $p\in]1,\infty[$ and $s_1\in ]-2/p,2/p]$.
Let $f\in \cB^{s_1}_{p}$ and
$a\in \cB^{\f2p}_{p}$ with $\|a\|_{\cB^{\f2p}_p}$ being
sufficiently small. Then the following equation
\begin{equation}\label{eqrlemp}
\dive\bigl((1+a)\nabla\Pi-f\bigr)=0.
\end{equation}
has a unique  solution $\nabla\Pi$  satisfying
\begin{equation}\label{S3eq6}
\|\nabla\Pi\|_{\cB^{s_1}_p}\lesssim\|f\|_{\cB^{s_1}_p}
\andf\|(1+a)\nabla\Pi\|_{\cB^{s_1}_p}\lesssim\|f\|_{\cB^{s_1}_p}.
\end{equation}
}\end{lem}
\begin{proof}
We can rewrite \eqref{eqrlemp} as
$$\D\Pi=-\dive(a\nabla\Pi)+\dive f.$$
Then  applying the operator $\nabla\D^{-1}$ to the above identity gives
\begin{equation}\label{S3eq6a} \nabla\Pi=\cM_a(\nabla\Pi)+\nabla\D^{-1}\dive f\with
\cM_a(g)\eqdef\nabla\D^{-1}\dive(ag).
\end{equation}
In view of the choice of $s_1$, we can use Lemma \ref{lemproductlaw} to get
$$\|\cM_a(g)\|_{\cB^{s_1}_p}\leq C\|a\|_{\cB^{\f2p}_{p}}
\|g\|_{\cB^{s_1}_p},\quad \mbox{i.e.}\quad
\|\cM_a\|_{\cL(\cB^{s_1}_p)}\leq C\|a\|_{\cB^{\f2p}_{p}},$$
where $\cL(\cB^{s_1}_p)$ means linear operator in $\cB^{s_1}_p$.
Thus if $\|a\|_{\cB^{\f2p}_p}$ is sufficiently small,
then $({\rm Id}-\cM_a)^{-1}$ is well-defined as an element of $\cL(\cB^{s_1}_p)$ with
$$({\rm Id}-\cM_a)^{-1}\leq\bigl(1-C\|a\|_{\cB^{\f2p}_{p}}\bigr)^{-1}.$$
As a result, we can solve a unique $\nabla\Pi$ from \eqref{S3eq6a} satisfying
$$\|\nabla\Pi\|_{\cB^{s_1}_p}
\lesssim\bigl(1-C\|a\|_{\cB^{\f2p}_{p}}\bigr)^{-1}
\|\nabla\D^{-1}\dive f\|_{\cB^{s_1}_p}
\lesssim\|f\|_{\cB^{s_1}_p}.$$
This completes the proof of this lemma.
\end{proof}

With Lemma \ref{lemp} at hand,
now we can present the proof of Lemma \ref{lem4.1}.

\begin{proof}[{\bf Proof of Lemma \ref{lem4.1}}]
For any constant $\mu>0,$ let us denote
\begin{equation}\label{lem4.1c}
w_{\mu}(t)\eqdef w(t)\exp\Bigl(-\mu\int_0^t \Upsilon(t')\,dt'\Bigr),
\quad\text{where}\quad\Upsilon(t)\eqdef\|b_1(t)\|_{\cB^{\f2p}_{p}}^2
+\|\nablah b_2(t)\|_{\cB^{\f2p}_{p}}.
\end{equation}
Multiplying $\exp\bigl(-\mu\int_0^t \Upsilon(t')\,dt'\bigr)$
and applying $\dhk$ to \eqref{eqt2dw} gives
\begin{equation}\label{eqtwmu}
\pa_t \dhk w_{\mu}+\mu \Upsilon(t)\dhk w_{\mu}-\dhk\Dh w_{\mu}
+\dhk F_\mu=0,
\end{equation}
where
$$F_\mu=b_1\cdot\nablah w_\mu
+w_\mu\cdot\nablah b_2
-a\D w_\mu+(1+a)\na\Pi_\mu-\eta_\mu.$$

Taking $L^2$ inner product of \eqref{eqtwmu}
with $|\dhk w_\mu|^{p-2}\dhk w_\mu$ yields
$$\f1p\f{d}{dt}\|\dhk w_\mu\|_{L^p}^p+\mu\Upsilon(t)\|\dhk w_\mu\|_{L^p}^p
-\int_{\R^3}\bigl(\dhk\D w_\mu-\dhk F_\mu\bigr)
\cdot|\dhk w_\mu|^{p-2}\dhk w_\mu\,dx=0.$$
It follows from  Lemma \ref{lemdan} that
$$-\int_{\R^2}\dhk \D w_\mu\cdot|\dhk w_\mu|^{p-2}\dhk w_\mu\,dx
\gtrsim2^{2j}\|\dhk w_\mu\|_{L^p}^p.$$
As a result, it comes out
$$\f1p\f{d}{dt}\|\dhk w_\mu\|_{L^p}^p
+\bigl(\mu\Upsilon(t)+2^{2j}\bigr)\|\dhk w_\mu\|_{L^p}^p
\lesssim\|\dhk F_\mu\|_{L^p}\|\dhk w_\mu\|_{L^p}^{p-1}.$$
Then for any $\e>0,$ multiplying the above inequality by
$\bigl(\|\dhk w_\mu\|_{L^p}^p+\e\bigr)^{\f1p-1}$ gives
\begin{align*}
\f{d}{dt}\bigl(\|\dhk& w_\mu\|_{L^p}^p+\e\bigr)^{\f1p}
+\bigl(\mu\Upsilon(t)+2^{2j}\bigr)
\bigl(\|\dhk w_\mu\|_{L^p}^p+\e\bigr)^{\f1p}\\
&=\bigl(\|\dhk w_\mu\|_{L^p}^p+\e\bigr)^{\f1p-1}\cdot
\Bigl(\f1p\f{d}{dt}\|\dhk w_\mu\|_{L^p}^p
+\bigl(\mu\Upsilon(t)+2^{2j}\bigr)\bigl(\|\dhk w_\mu\|_{L^p}^p+\e\bigr)\Bigr)\\
&\lesssim\bigl(\|\dhk w_\mu\|_{L^p}^p+\e\bigr)^{\f1p-1}
\cdot\Bigl(\|\dhk F_\mu\|_{L^p}\|\dhk w_\mu\|_{L^p}^{p-1}+\bigl(\mu\Upsilon(t)+2^{2j}\bigr)\e\Bigr)\\
&\lesssim\|\dhk F_\mu\|_{L^p}
+\bigl(\mu\Upsilon(t)+2^{2j}\bigr)\e^{\f1p}.
\end{align*}
Integrating the above inequality over $[0,t]$ gives
\begin{equation}\begin{split}\label{eqtDj}
\bigl(\|\dhk w_\mu(t)\|_{L^p}^p&+\e\bigr)^{\f1p}
+\int_0^t\bigl(\mu\Upsilon(t')+2^{2j}\bigr)
\bigl(\|\dhk w_\mu(t')\|_{L^p}^p+\e\bigr)^{\f1p}\,dt'\\
&\lesssim \bigl(\|\dhk w_0\|_{L^p}^p+\e\bigr)^{\f1p}
+\|\dhk F_\mu\|_{L^1_t(L^p)}
+\bigl(\mu\|\Upsilon\|_{L^1_t}+2^{2j}t\bigr)\e^{\f1p}.
\end{split}\end{equation}
Then by letting $\e\rightarrow0^+$, and multiplying the resulting inequality by $2^{js}$,
and then summing up in $j\in\Z$, we obtain
\begin{equation}\begin{split}\label{lem4.1.1}
\|w_\mu\|_{\wt{L}^\infty_t(\cB^s_{p})}
+\mu\|w_\mu\|_{L^1_{t,\Upsilon}(\cB^s_{p})}
+\|\Dh w_\mu\|_{L^1_t(\cB^s_{p})}
\leq C\|w_0\|_{\cB^s_{p}}
+C\|F_\mu\|_{L^1_{t}(\cB^s_{p})}.
\end{split}\end{equation}

For the terms in $\|F_\mu\|_{L^1_{t}(\cB^s_{p})}$,
we first get, by using Lemma \ref{lemproductlaw} that
\begin{equation}\begin{split}\label{lem4.1.2}
C\|b_1\cdot\nablah w_\mu\|_{L^1_t(\cB^s_{p})}
&\leq C\int_0^t\|b_1\|_{\cB^{\f2p}_{p}}
\|\nablah w_\mu\|_{\cB^s_{p}}\,dt'\\
&\leq \int_0^t\Bigl(\f18\|\Dh w_\mu\|_{\cB^s_{p}}
+C\|w_\mu\|_{\cB^s_{p}}\|b_1\|_{\cB^{\f2p}_{p}}^2\Bigr)\,dt'.
\end{split}\end{equation}
Similarly, we have
\begin{equation}
C\|w_\mu\cdot\nablah b_2\|_{L^1_t(\cB^s_{p})}
\leq C\int_0^t\|w_\mu\|_{\cB^s_{p}}
\|\nablah b_2\|_{\cB^{\f2p}_{p}}\,dt'.
\end{equation}

On the other hand, due to $\|a\|_{{L}^\infty_t(\cB^1_{2})}$ is sufficiently small,
we have
\begin{equation}\label{lem4.1.3}
C\|a\Dh w_\mu\|_{L^1_t(\cB^s_{p})}
\leq C\|a\|_{L^\infty_t(\cB^{1}_{2})}\|\Dh w_\mu\|_{L^1_t(\cB^s_{p})}
\leq\f18\|\Dh w_\mu\|_{L^1_t(\cB^s_{p})}.
\end{equation}
While by taking space divergence to the $w$ equation of \eqref{eqt2dw}, we find
$$\dive\bigl((1+a)\na\Pi-a\D w+b_1\cdot\nablah w
+w\cdot\nablah b_2-\eta\bigr)=0.$$
Then by applying Lemma \ref{lemp} and the estimates
\eqref{lem4.1.2}-\eqref{lem4.1.3}, we deduce
\begin{equation}\begin{split}\label{lem4.1.4}
C\bigl\|(1+a)\nablah\Pi_\mu\bigr\|_{L^1_t(\cB^s_{p})}
&\leq C\bigl(\|b_1\cdot\nablah w_\mu\|_{L^1_t(\cB^s_{p})}
+\|w_\mu\cdot\nablah b_2\|_{L^1_t(\cB^s_{p})}\\
&\qquad+\|a\Dh w_\mu\|_{L^1_t(\cB^s_{p})}
+\|\eta_\mu\|_{L^1_t(\cB^s_{p})}\bigr)\\
&\leq \f14\|\Dh w_\mu\|_{L^1_t(\cB^s_{p})}
+C\|w_\mu\|_{L^1_{t,\Upsilon}(\cB^s_{p})}
+C\|\eta_\mu\|_{L^1_t(\cB^s_{p})}.
\end{split}\end{equation}

Now by inserting the estimates \eqref{lem4.1.2}-\eqref{lem4.1.4}
into \eqref{lem4.1.1}, we finally achieve
\begin{align*}
\|w_\mu\|_{\wt{L}^\infty_t(\cB^s_{p})}
+(\mu-C)\|w_\mu\|_{L^1_{t,\Upsilon}(\cB^s_{p})}
+\|\Dh w_\mu\|_{L^1_t(\cB^s_{p})}
\lesssim\|w_0\|_{\cB^s_{p}}
+\|\eta_\mu\|_{L^1_t(\cB^s_{p})}.
\end{align*}
Taking $\mu$ to be sufficiently large gives
$$\|w_\mu\|_{\wt{L}^\infty_t(\cB^s_{p})}
+\|\Dh w_\mu\|_{L^1_t(\cB^s_{p})}
\lesssim\|w_0\|_{(\cB^s_{p})}
+\|\eta\|_{L^1_{t}(\cB^s_{p})},$$
which together with the expression \eqref{lem4.1c} ensures
the first desired estimate \eqref{lem4.1a}.
\smallskip

As for the case when $b_2=\eta=0$ and $s\in]-2/p,1/p]$,
in order to prove \eqref{lem4.1b},
we only need to modify the $\Upsilon(t)$
in \eqref{lem4.1c} to be
$\|b_1\|_{L^2}^{\f2{p-1}}
\|\nablah b_1\|_{L^2}^2,$
and the estimate \eqref{lem4.1.2} to be
\begin{align*}
C \|b_1\cdot\nablah w_\mu\|_{L^1_t(\cB^s_{p})}
&\leq C\int_0^t\|b_1\|_{(\cB^{\f1p}_{p})}
\|\nablah w_\mu\|_{(\cB^{s+\f1p}_{p})}\,dt'\\
&\leq C \int_0^t\|b_1\|_{L^2}^{\f1p}
\|\nablah b_1\|_{L^2}^{1-\f1p}
\|w_\mu\|_{\cB^{s}_p}^{\f12-\f1{2p}}
\|\Dh w_\mu\|_{\cB^{s}_p}^{\f12+\f1{2p}}\,dt'\\
&\leq \int_0^t\Bigl(\f18\|\Dh w_\mu\|_{\cB^s_{p}}
+C\|w_\mu\|_{\cB^s_{p}}\|b_1\|_{L^2}^{\f2{p-1}}
\|\nablah b_1\|_{L^2}^2\Bigr)\,dt'.
\end{align*}
This completes the proof of this lemma.
\end{proof}

\medskip

\noindent {\bf Acknowledgments.}
Y. Liu is supported by NSF of China under grant 12101053.

\end{document}